\documentclass[12pt]{amsart}

\usepackage{fourier}
\usepackage{calrsfs}

\usepackage{amssymb}
\usepackage{amsfonts}

\usepackage{color}
\usepackage{epsfig}
\usepackage{graphicx}
\usepackage{graphics}
\usepackage{amssymb}
 \usepackage{fullpage}
\usepackage{comment}

\newfont{\vlte}{eufm10 at 22pt}
\newfont{\lte}{eufm10 at 18pt}
\newfont{\smb}{msbm6}
\newfont{\mmb}{msbm8}
\newfont{\tmb}{msbm10}
\newfont{\lmb}{msbm10 at 18pt}

\newcommand{\ord}{\mbox{ord}}

\newcommand{\be}{\begin{enumerate}}
\newcommand{\ee}{\end{enumerate}}

\newcommand{\denom}{\mbox{denom}}

\chardef\secsym=129
\newcommand{\calA}{{\mathcal A}}

\newcommand{\calP}{{\mathcal P}}

\newcommand{\calS}{{\mathcal S}}
\newcommand{\calT}{{\mathcal T}}

\newcommand{\calV}{{\mathcal V}}
\newcommand{\calW}{{\mathcal W}}

\newcommand{\Q}{{\mathbb Q}}
\newcommand{\R}{{\mathbb R}}
\newcommand{\Z}{{\mathbb Z}}

\newcommand{\pp}{{\mathfrak p}}

\chardef\sha=88
\newtheorem{theorem}{Theorem}[section]
\newtheorem{lemma}[theorem]{Lemma}
\newtheorem{corollary}[theorem]{Corollary}
\newtheorem{proposition}[theorem]{Proposition}

\theoremstyle{definition}
\newtheorem{definition}[theorem]{Definition}
\newtheorem{question}[theorem]{Question}
\newtheorem{conjecture}[theorem]{Conjecture}
\newtheorem{example}[theorem]{Example}

\theoremstyle{remark}
\newtheorem{remark}[theorem]{Remark}

  \theoremstyle{plain}

\begin{document}
\bibliographystyle{plain}%
 \title{Defining Integers}%
\author{Alexandra Shlapentokh}%
\thanks{The author has been partially supported by NSF grant DMS-0650927 and by a grant from John Templeton Foundation.}
\address{Department of Mathematics \\ East Carolina University \\ Greenville, NC 27858, USA}%
\email{shlapentokha@ecu.edu }
\urladdr{www.personal.ecu.edu/shlapentokha} 
\maketitle

\noindent You can't always get what you want\\
You can't always get what you want\\
You can't always get what you want\\
But if you try sometimes you might find\\
You get what you need \\
\vskip 0.1in
Rolling Stones

\section{Prologue}
\subsection{A Question and the Answer}
The history of the problems discussed in this exposition goes back to a question that was posed by Hilbert in 1900: is there an algorithm which, given arbitrary polynomial equation in several variables over $\Z$, determines whether such an equation has solutions in $\Z$? (At the time Hilbert posed the question, a rigorous notion of the algorithm did not yet exist.  So the version above is a modern interpretation of the question.)  This
question, being the tenth question on a list, became known as Hilbert's Tenth Problem (referred to as ``HTP'' in the future), and was answered negatively in the work of
M. Davis, H. Putnam, J. Robinson and Yu. Matijasevich. (See \cite{Da1}, \cite{Da2} and \cite{Mate}.)  In fact, as we explain below, significantly more was shown: it was proved that recursively enumerable subsets of integers and Diophantine subsets of integers were the same.  We define these sets below.
\begin{definition}[Recursive and Recursively Enumerable Subsets of $\Z$]%
 A set $A \subseteq \Z^m$ is called \emph{recursive, computable or decidable} if there is an
algorithm (or a computer program) to determine the membership in the set.%

A set $A \subseteq \Z^m$ is called \emph{recursively or computably  enumerable } if there is an algorithm (or a computer program)
to list the set.%
\end{definition}%
The following theorem  is a well-known result from Recursion Theory (see for example \cite{Rogers}[\S 1.9]).
\begin{theorem}%
There exist recursively enumerable sets which are not recursive.%
\end{theorem}%
We now define Diophantine sets in a somewhat more general setting.
\begin{definition}[Diophantine Sets: a Number-Theoretic Definition]
Let $R$ be a  commutative ring with identity.  (All the rings considered below satisfy these assumptions.)   A subset $A \subset R^m$ is called Diophantine over $R$ if there exists a polynomial $p(T_1,\ldots T_m,X_1,\ldots,X_k)$ with coefficients in $R$ such that for any element
$(t_1,\ldots,t_m) \in R^m$ we have that
\[%
\exists x_1,\ldots, x_k \in \Z: p(t_1,\ldots,t_m,x_1,\ldots,x_k) = 0
\]%
\[%
\bigm\Updownarrow
\]%
\[%
 (t_1,\ldots,t_m) \in A.
\]%
In this case we call $p(T_1,\ldots,T_m,X_1,\ldots,X_k)$ a \emph{Diophantine definition} of $A$ over $R$.
\end{definition}
\begin{remark}
 Diophantine sets can also be described as the sets \emph{existentially definable in $R$} in the language of rings or as \emph{projections of algebraic sets}.
 \end{remark}
Given the MDRP result we immediately obtain  the following important corollary.
\begin{corollary}
There are undecidable Diophantine subsets of $\Z$.
\end{corollary}%
It is easy to see that the existence of undecidable Diophantine sets implies that no algorithm as requested by Hilbert exists.
 Indeed, suppose $A\subset \Z$ is an undecidable Diophantine set with a Diophantine definition $P(T,X_1,\ldots,X_k)$.  Assume also that  we have an algorithm to determine the existence of integer solutions for polynomials. Now, let $a \in \Z$  and observe that $a \in A$ if and only if $P(a,X_1,\ldots,X_K)=0$ has solutions in $\Z^k$. So if we can answer Hilbert's question effectively, we can determine the membership in $A$ effectively.

It is also not hard to see that Diophantine sets are recursively enumerable. Given a polynomial $p(T,\bar{X})$ we can
effectively list all $t \in \Z$ such that $p(t,\bar{X})=0$ has a solution $\bar{x} \in \Z^k$ in the
following fashion. Using a recursive listing of $\Z^{k+1}$, we can plug each $(k+1)$-tuple into
$p(T,\bar{X})$ to  see if the value is $0$. Each time we get a zero we add the first element of the $(k+1)$-tuple to the $t$-list.
\subsection{Some Easy Facts or Getting Better Acquainted}
A Diophantine set does not have to be complicated: one of the simplest Diophantine sets is the set of even integers  $$\{t \in \Z | \exists w \in \Z: t =2w\}.$$
To construct more complicated examples we need to establish some properties of Diophantine sets.
\begin{lemma}%
Intersections and unions of Diophantine sets over $\Z$ are Diophantine.
\end{lemma}%
\begin{proof}%
Suppose $P_1(T,\bar X), P_2(T,\bar Y)$ are Diophantine definitions of subsets $A_1$ and $A_2 $ of $\Z$
respectively over $\Z$. In this case
\[%
P_1(T,\bar X) P_2(T,\bar Y)
\]%
is a Diophantine definition of $A_1 \cup A_2$, \\
and
\[%
P_1^2(T,\bar X)+ P_2^2(T,\bar Y)
\]%
is a Diophantine definition of $A_1 \cap A_2$.
\end{proof}%
The fact that over $\Z$ an intersection of Diophantine sets  is Diophantine is related to another important aspect of Diophantine equations over $\Z$: a finite system of equations is always equivalent to a single equation in the sense that both the system and the equation have the same solutions over $\Z$ and in the sense that given a finite system of equations, the equivalent equation can be constructed effectively. We prove this assertion in the lemma below.
\begin{lemma}[Replacing Finitely Many by One over $\Z$]%
Any finite system of equations over $\Z$ can be \emph{effectively} replaced by a single polynomial equation
over $\Z$ with identical $\Z$-solution set.%
\end{lemma}%
\begin{proof}%
Consider a system of equations
\[%
\left \{%
\begin{array}{c}
g_1(x_1,\ldots,x_k)=0\\
g_2(x_1,\ldots,x_k)=0\\
\ldots\\
g_m(x_1,\ldots,x_k)=0
\end{array}
\right .
\]%

This system has solutions in $\Z$ if and only if the following equation has solutions in $\Z$:
\[%
g_1(x_1,\ldots,x_k)^2  + g_2(x_1,\ldots,x_k)^2 + \ldots + g_m(x_1,\ldots,x_k)^2= 0
\]%
\end{proof}%
In fact we can replace a finite system of polynomial equations by an equivalent single polynomial equation over any integral domain $R$ whose fraction field is not algebraically closed.
\begin{lemma}[Replacing Finitely Many by One over an Arbitrary Integral Domain ]%
Let $R$ be any ring such that its fraction field $K$ is not algebraically closed.  In this case, any finite system of equations over $R$ can be replaced by a single polynomial equation over $R$ with  identical $R$-solution set.
\end{lemma}%
\begin{proof}%
It is enough to consider the case of two equations: $f(x_1,\ldots,x_n)=0$ and $g(x_1,\ldots,x_n) = 0$.  If $h(x) = \sum_{i=0}^ka_i x^i$ is a polynomial over $R$ without any roots in $K$, then the polynomial
\[
 \sum_{i=0}^ka_i f^i(x_1,\ldots,x_n) g^{k-i}(x_1,\ldots,x_n)=0
\]
has solutions in $K$ if and only if $f(x_1,\ldots,x_n)=0$ and $g(x_1,\ldots,x_n) = 0$ have a common solution in $K$.
\end{proof}%
\begin{remark}
If the ring in question is recursive, i.e. there exists an injective map from the ring into $\Z$ such that the image of the ring is recursive and the ring operations are translated by recursive functions (functions whose graphs are recursive sets), then there exists a recursive function which can take the coefficients of the system equations as its inputs and output the coefficients of the corresponding single equation with the same solution set over the ring.  For all the rings we consider in this exposition the construction of such a function, which would depend on a construction of an irreducible polynomial over the fraction field, is fairly straightforward.  However, in general, the situation can be a lot more complicated.
\end{remark}
We can use this property of finite systems to give more latitude to our  Diophantine definitions
\begin{corollary}%
We can let the Diophantine definitions over $\Z$ consist of several polynomials without changing the nature of
the relation.
\end{corollary}%
One surprisingly useful tool for writing Diophantine definitions has to do with an elementary property of GCD's (greatest common divisors).
\begin{proposition}
If $a, b \in \Z_{\not = 0}$ with $(a,b)=1$ then there exist $x, y \in \Z$ such that $ax + by =1$.
\end{proposition}
The GCD's can be used to show that the set of non-zero integers is Diophantine and thus allow us to require that values of variables are not equal, as well as to perform ``division'' as will be shown later.  On a more theoretical level we can say that the \emph{positive} existential theory of $\Z$ is the same as the existential theory of $\Z$.
\begin{proposition}%
The set of non-zero integers has the following Diophantine definition over $\Z$:
\[
\{t \in \Z|\exists x,u,v\in \Z: (2u -1)(3v-1)=tx\}
\]
\end{proposition}%
\begin{proof}%
If $t=0$, then either $2u-1=0$ or $3v-1=0$ has a solution in $\Z$, which is impossible.
Suppose now $t \not=0$. Write $t=t_2t_3$, where $t_2$ is odd and $t_3 \not \equiv 0 \mod 3$.  Since
$(t_2,2)=1$ and $(t_3,3)=1$, by a property of GCD  there exist $u, x_u, v, x_v \in \Z$ such
that $$2u + t_2x_u=1$$ and $$3v + t_3x_v=1.$$   Thus $(2u-1)(3v-1)=t_2x_ut_3x_v=t(x_ux_v)$.
\end{proof}%
Another important Diophantine definition allows us to convert inequalities into equations.
\begin{lemma}[Diophantine definition of the set of non-negative integers]
From Lagrange's Theorem we get the following representation of non-negative integers:
\[
\{ t \in \Z | \exists x_1, x_2, x_3, x_4: t=x_1^2 + x_2^2 +x_3^2 +x_4^2 \}
\]
\end{lemma}
\subsection{Becoming More Ambitious}%
The questioned posed by Hilbert about the ring of integers can of course be asked about any recursive ring $R$: is there an algorithm, which if given an arbitrary polynomial equation in several variables with coefficients in $R$, can determine whether this equation has solutions in $R$?  Arguably, the most prominent open questions in the area are the decidability of an analog of Hilbert's Tenth Problem for $R=\Q$ and $R$ equal to the ring of integers of an arbitrary number field.  The recent developments concerning these two problems are the main subjects of this exposition.

Before we proceed further with our discussion of HTP over $\Q$ we would like to point out that it is not hard to  see that decidability of HTP over $\Z$ would imply decidability of HTP for $\Q$. Indeed, suppose we knew how to determine whether solutions exist over $\Z$.  If $Q(x_1,\ldots,x_k)$ is a polynomial with integer coefficients, then
\[
\exists x_1,\ldots,x_k \in \Q: Q(x_1,\ldots,x_k)=0
\]
\[
\bigm\Updownarrow
\]
\[
\exists y_1,\ldots,y_k, z_1, \ldots, z_k \in \Z: Q(\frac{y_1}{z_1},\ldots,\frac{y_k}{z_k})=0 \land z_1\ldots z_k \not =0,
\]
where we remind the reader we can rewrite $z_1\ldots z_k \not =0$ as a polynomial equation and convert the resulting finite system of equations into a single one. So if we can determine whether the resulting equation had solutions over $\Z$, we can determine whether the original equation had solutions over $\Q$.  Unfortunately, the reverse implication does not work: we don't know of any easy way to derive the undecidability of HTP over $\Q$ from the analogous result over integers.  As a matter of fact, as we will see below we don't know of any way  of deriving the undecidability of HTP over $\Q$ (we are not even sure the problem is undecidable over $\Q$).

One of the earliest methods suggested for showing that HTP was undecidable over $\Q$ used Diophantine definitions.  This idea can be summarized in the following lemma:
\begin{lemma}%
If $\Z$ has a Diophantine definition $p(T,\bar X)$ over $\Q$, then HTP is not decidable over $\Q$.
\end{lemma}%
\begin{proof}%
Let $h(T_1,\ldots,T_l)$ be a polynomial with rational integer coefficients and consider the following system of
equations.
\begin{equation} \label{eq:1} \left \{%
\begin{array}{c} h(T_1,\ldots,T_l) = 0\\%
p(T_1,\bar X_1)=0\\%
\ldots\\%
p(T_l,\bar X_l)=0%
\end{array}%
\right .%
\end{equation}%
It is easy to see that $h(T_1,\ldots,T_l) = 0$ has solutions in $\Z$ if and only if system (\ref{eq:1}) has solutions in $\Q$. Thus if
HTP is decidable over $\Q$, it is decidable over $\Z$.
\end{proof}%
Unfortunately, the Diophantine definition plan quickly ran into problems.
\section{Complications}
\subsection{Some Unpleasant Thoughts}
In 1992 Barry Mazur formulated a series of conjectures which were to play an important role in the development of the subject (see \cite{M1}, \cite{M2}, \cite{M4}).  Before we state one of these conjectures, we need a definition.
\begin{definition}[Affine Algebraic Sets and Varieties.]
If $\{p_1(x_1,\ldots,x_m),\ldots, p_k(x_1,\ldots,x_m)\}$ is a finite set of polynomial equations over some field $K$, then the set of common zeros of these polynomials in $K^m$ is  called an algebraic set.  An algebraic set which is irreducible, i.e. is not a union of non-empty algebraic sets, is called a variety.
\end{definition}
Mazur's conjectures on the topology of rational points  is stated below:
\begin{conjecture}[Topology of Rational Points]
\label{conjecture:1}
Let $V$ be any variety over $\Q$. Then the topological closure of $V(\Q)$ in $V(\R)$ possesses at most a finite number of connected components.%
\end{conjecture}
This conjecture had an unpleasant consequence.
\begin{conjecture}
There is no Diophantine definition of $\Z$ over $\Q$.
\end{conjecture}

Mazur's conjecture also refers to projective varieties, but it is the affine variety case which has the most consequences for HTP over $\Q$.  We should also note that one can replace ``variety'' by ``algebraic set'' without changing the scope of the conjecture.  (See Remark 11.1.2 of \cite{Sh34}.)     As a  matter of  fact, if Conjecture \ref{conjecture:1} is true, no infinite and discrete (in the archimedean topology) set has a Diophantine definition over $\Q$.  The affine version of Mazur's conjecture can be thought of in the following manner.  Suppose you are given a system of polynomial equations:
\begin{equation}%
\label{sys:Maz}
\left \{%
\begin{array}{c}%
P_1(x_1,\ldots,x_k)=0\\
P_2(x_1,\ldots,x_k)=0\\
\ldots \\
P_m(x_1,\ldots,x_k)=0\\
\end{array}%
\right .
\end{equation}%
Think of solutions to this system as points in $\R^k$ but consider only the points whose coordinates are rational numbers.  In
other words we are interested in the set
\[%
RP=\{(x_1,\ldots,x_k) \in \Q^k: (x_1,\ldots,x_k) \mbox{ is a solution to system (\ref{sys:Maz})}\}.
\]%

Now take the topological closure of $RP$ in $\R^k$ (i.e. the points plus the ``boundary'').  Mazur's conjecture asserts that the resulting set will have finitely many ``connected pieces'' also known as connected components. It is the finite number of these components that precludes Diophantine definability of infinite discrete subsets.

\subsection{Introducing New Models}
Since the plan to construct a Diophantine definition of $\Z$ over $\Q$ ran into substantial difficulties, alternative ways were considered for showing that HTP had no solution over $\Q$.  One of the alternative methods required construction of a Diophantine model of $\Z$.
\begin{definition}[Diophantine Model of $\Z$]
Let $R$ be a recursive ring whose fraction field is not algebraically closed and let $\phi : \Z \longrightarrow R^k$ be a
recursive injection mapping Diophantine sets of $\Z$ to Diophantine sets of $R^k$. Then $\phi$ is called a Diophantine
model of $\Z$ over $R$.
\end{definition}
\begin{remark}[An Alternative Terminology from Model Theory]
Model theorist have an alternative terminology for a map described above.  They would translate the statement  that $R$ has a Diophantine model of $\Z$ as $\Z$ being existentially definably interpretable in $R$.  (See Chapter 1, Section 3 of \cite{Marker}.)
\end{remark}
It is not hard to see that sending Diophantine sets to Diophantine sets makes the map automatically recursive.  The recursiveness of the map follows from the fact that the $\phi$-image of the graph of addition is Diophantine and therefore is recursively enumerable (by the same argument as over $\Z$). Thus, we have an effective listing of the set
\[%
D_+=\{(\phi(m),\phi(n),\phi(m+n)), m, n \in \Z\}.
\]%
Assume we have computed $\phi(r-1)$ for some positive integer $r$. Now start listing $D_{+}$ until we come across a triple whose first two entries are $\phi(r-1)$ and $\phi(1)$. The third element of the triple must be $\phi(r)$.  We can simplify the requirements for the map further.
\begin{proposition}
If $R$ is a recursive ring and  $\phi: \Z \longrightarrow R^k$ is injective for some $k \in \Z_{>0}$,  then $\phi$ is a Diophantine model if and only if the images of the graphs of $\Z$-addition and $\Z$-multiplication are Diophantine over $R$.
\end{proposition}
This proposition  can be proved by a straightforward induction argument which we do not reproduce here.

It quite easy to see that the following proposition holds.
\begin{proposition}
If $R$ is a recursive ring with a Diophantine model of $\Z$, then HTP has no solution over $R$.
\end{proposition}
\begin{proof}
If $R$ has a Diophantine model of $\Z$, then $R$ has undecidable Diophantine sets, and the existence of undecidable Diophantine sets over $R$ leads us to the undecidability of HTP over $R$ in the same way as it happened over $\Z$.   To show that $R$ has undecidable Diophantine sets,  let $A \subset \Z$
be an undecidable Diophantine set and suppose we want to determine whether an integer $n \in A$.  Instead of answering
this question directly we can ask whether $\phi(n) \in \phi(A)$.  By assumption $\phi(n)$ is algorithmically computable.
 So if $\phi(A)$ is a computable subset of $R$, we have a contradiction.
\end{proof}%
It now follows that constructing a Diophantine model of $\Z$ over $\Q$ will solve our problem.
\subsection{A Steep Curve}
\label{subsec:oldplan}
An old plan for building a Diophantine model of $\Z$ over $\Q$ involved using elliptic curves.
Consider an equation of the form:   \begin{equation} \label{elliptic} y^2=x^3+ax+b,\end{equation} where $a, b \in \Q$ and $\Delta = -16(4a^3 + 27b^2) \not =0$.  This equation defines an elliptic curve  (a non-singular plane curve of genus 1).   Figure \ref{fig:1}  is the graph of such an elliptic curve $y^2=x^3-2x+1$ generated using Maple.
\begin{figure}
 \includegraphics[scale=0.5]{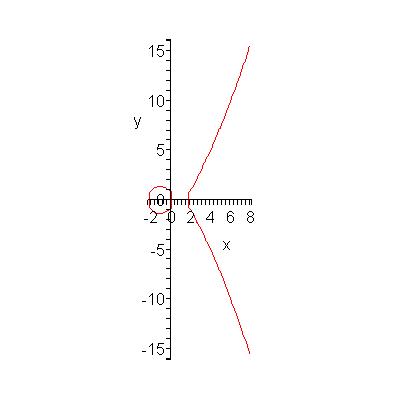}
  \caption{} \label{fig:1}
\end{figure}

  All the  points $(x,y) \in \Q^2$ satisfying \eqref{elliptic} (if any) together with $O$ -- the ``point at infinity'' form an abelian group, i.e. there is a way to define addition on the points of an elliptic curve with  $O$ serving as the identity. The group law on an elliptic curve can be represented geometrically (see for example \cite{Sil1}[Chapter III, \S 3]).  However, what is important to us is the algebraic representation of the group law.
Let $P=(x_P,y_P), Q=(x_Q,y_Q), R=(x_R,y_R)$ be the points on an elliptic curve $E$ with rational coordinates.  If $P+_E Q=R$ and $P, Q, R \not = O$, then
$x_R=f(x_P, y_P, x_Q, y_Q), y_R=g(x_P, y_P, x_Q, y_Q)$, where $f(z_1,z_2, z_3, z_4), g(z_1,z_2,z_3,z_4)$ are fixed (somewhat unpleasant looking) \emph{rational functions}.  Further, $-P =(x_P, -y_P)$.   Mordell-Weill Theorem (see \cite{Sil1}[Chapter III]) tells us that the abelian group formed by points of an elliptic curve over $\Q$ is finitely generated, meaning it has a finite rank and a finite torsion subgroup. It is also not very difficult to find elliptic curves whose rank is one. So let $E$ be such an elliptic curve defined over $\Q$ such that $E(\Q) \cong \Z$ as abelian groups. (In other words $E(\Q)$ has no torsion points.  In practice torsion points are not an impediment, but they do complicate the discussion.)  Let $P_1$ be a generator and consider a map sending an integer $n \not = 0$ to $[n]P=(x_n,y_n)$.  (We should also take care of 0, but we will ignore this issue for the moment.) The group law assures us that under this map \emph{the image of the graph of addition is Diophantine}.  Unfortunately,   it is not clear what happens to \emph{the image of the graph of multiplication}.  Nevertheless one might think that we have a starting point at least for our Diophantine model of $\Z$.  Unfortunately, it turns out that situation with Diophantine models is not any better than with Diophantine definitions.
\subsection{More Bad News}
A new piece of bad news came in the guise of a theorem of Cornelissen and Zahidi (see \cite{CZ}).
\begin{theorem}%
If Mazur's conjecture on topology of rational points holds, then there is no Diophantine model of $\Z$ over $\Q$.
\end{theorem}
This theorem left HTP over $\Q$ completely unapproachable.  It is often the case that faced with  intractable difficulties, mathematicians escape by changing the problem. Depending on one's point of view of the subject one could consider changing the problem in different ways: a number theorist might consider changing the object or possibly summon a big conjecture for assistance; a logician might reconsider the ban on all universal quantifiers.  Perhaps one or two should be allowed back.  As it turned out all these paths were explored. Below we describe these new problems starting with the ones which considered alternate objects.

\section{Big and Small}
\subsection{The Rings between $\Z$ and $\Q$}
\label{subsec:between}
The new objects which were introduced into the subject were the rings ``in between'' $\Z$ and $\Q$.
\begin{definition}[A Ring in between]
Let $\calS$ be a set of primes of $\Q$.  Let $O_{\Q,\calS}$ be the following subring of $\Q$.
\[%
\left \{\frac{m}{n}: m, n \in \Z, n \not = 0, n \mbox{ is divisible by primes of } \calS
\mbox{ only }\right \}
\]%
If $\calS = \emptyset$, then $O_{\Q,\calS}=\Z$.  If $\calS$ contains all the primes of $\Q$,  then
$O_{\Q,\calS}=\Q$.  If $\calS$ is finite, we call the ring \emph{small}.  If $\calS$ is infinite, we
call the ring \emph{large or big}, and if the natural density of $\calS$ is equal to 1,  we call the ring \emph{very large or very big}.%
\end{definition}%
Some of these rings have other (canonical) names: the small rings are also called rings of $\calS$-integers, and when $\calS$ contains all but finitely many primes, the rings are called semi-local subrings of $\Q$.   The definition of very large rings uses the notion of natural density of a prime set stated below.
\begin{definition}[Natural Density]%
If $\calA$ is a set of primes, then the natural density of $\calA$  is equal to the limit below (if it exists):
\[%
\lim_{X \rightarrow \infty}\frac{\#\{p \in A, p\leq X\}}{\# \{p \leq X\}}
\]%
\end{definition}

The big and small rings are not hard to construct.

\begin{example}[A Small Ring not Equal to $\Z$]
\[%
\{\frac{m}{3^a5^b}: m \in \Z, a, b \in \Z_{>0}\}
\]%
\end{example}

\begin{example}[A Big Ring not Equal to $\Q$]
\[%
\{\frac{m}{\prod p_i^{n_i}}:  p_i \equiv 1 \mod 4, n_i \in \Z_{>0}\}
\]%
\end{example}

Given a big or a small ring $R$ we can now ask the following questions which were raised above with respect to $\Q$:
\begin{itemize}
\item Is HTP solvable over $R$?
\item Do integers have a Diophantine definition over $R$?
\item Is there a Diophantine model of integers over $R$?
\end{itemize}
While the answers to these questions are interesting on their own right, the hope (possibly unjustified) is that understanding the big rings will eventually lead us to $\Q$.
\subsection{Diophantine Properties of Big and Small Rings}%
Before trying to answer the questions above, one should observe  that the big and small rings share many Diophantine properties with the integers:
 \begin{proposition}%
 \label{prop:dioph}
 \begin{enumerate}%
 \item The set of non-zero elements of a big or a small ring is Diophantine over the ring.
\item ``One=finitely many'' over big and small rings.%
\item   The set of non-negative elements of a big or a small ring $R$ is Diophantine over $R$: a small modification of the Lagrange
argument  is required to accommodate possible denominators
\[%
\{ t \in R| \exists x_1, x_2, x_3, x_4, x_5: x_5^2t=x_1^2 + x_2^2 +x_3^2 +x_4^2 \land x_5 \not=0 \}
\]%
\end{enumerate}%
\end{proposition}%

It turned out that we already knew everything we needed to know about small rings from the work of Julia Robinson (see \cite{Rob1}).  In particular from her work on the first-order definability of integers over $\Q$ one can deduce the following theorem and corollary.
\begin{theorem}[Julia Robinson]
$\Z$ has a Diophantine definition over any small subring of $\Q$.
\end{theorem}
\begin{corollary}
HTP is unsolvable over all small subrings of $\Q$.
\end{corollary}

Over large rings the questions turned out to be far more difficult.

\section{A Different Model}
\subsection{Existential Model of $\Z$ over a Very Large Subring}
In 2003 Poonen in \cite{Po} proved the first result on Diophantine undecidability (unsolvability of HTP) over a big subring of $\Q$.
\begin{theorem}%
There exist recursive sets of  primes ${\mathcal T}_1$ and ${\mathcal T}_2$, both of natural density zero
and with an empty intersection, such that for any set ${\mathcal S}$ of   primes containing ${\mathcal
T}_1$ and avoiding  ${\mathcal T}_2$, the following hold:%
\begin{itemize} %
\item $\Z$ has a  Diophantine model over $O_{\Q,\calS}$.%
\item Hilbert's Tenth Problem is undecidable over $O_{\Q,\calS}$.
\end{itemize}
\end{theorem}%

Poonen used elliptic curves to prove his result but the model he constructed was very different from the one envisioned by the old elliptic curve plan we described earlier.  Poonen modeled integers \emph{by approximation}.  The construction of the model does start with an  elliptic curve of rank one
\begin{equation}
\label{eq:elliptic}
E: y^2=x^3+ax+b
\end{equation}
selected so that for a set of primes $\calS$, except possibly for finitely many points,  the only multiples of a generator $P$ that have their affine coordinates in the ring $O_{\Q,\calS}$ are in the sequence $[\pm \ell_i]P=(x_{\ell_i},\pm y_{\ell_i})$
 with \emph{$|y_{\ell_j} - j| < 10^{-j}$}. We remind the reader that we know how to define positive numbers using a variation on  Lagrange's theme (Proposition \ref{prop:dioph}) and  how to get rid of a finite set of undesirable values such as points of finite order (just say ``$\not =$'' as in Proposition \ref{prop:dioph} again).
We claim that $\phi: j \longrightarrow y_{\ell_j}$ is a Diophantine model of $\Z_{>0}$.  In other words we claim that $\phi$ is an  injection and the following sets are Diophantine:
\[%
D_+=\{(y_{\ell_i}, y_{\ell_j}, y_{\ell_k}) \in D^3: k = i+ j, k,i,j \in \Z_{>0}\}%
\]%
and%
\[%
D_2= \{(y_{\ell_i}, y_{\ell_k}) \in D^2: k = i^2, i \in \Z_{>0}\}.%
\]%
(Note that if $D_+$ and $D_2$ are Diophantine, then $D_{\times}=\{(y_{\ell_i}, y_{\ell_j}, y_{\ell_k}) \in D^3: k = i j, k,i,j \in \Z_{>0}\}$ is also Diophantine since $xy =\frac{1}{2}((x+y)^2-x^2-y^2$).)
It is easy to show that
\[%
k=i+j \Leftrightarrow |y_{\ell_i}+y_{\ell_j}-y_{\ell_k}| < 1/3.%
\]%
and with the help of Lagrange this makes $D_+$ Diophantine.  Similarly we have that
\[%
k=i^2 \Leftrightarrow  |y_{\ell_i}^2-y_{\ell_k}| < 2/5,%
\]%
implying that $D_2$ is Diophantine.

To restrict the number of solutions to  the elliptic curve equation, Poonen's construction relied to a large extent on the fact that the denominators of the coordinates of  points on an elliptic curve which are multiples of a single point form a divisibility sequence: an integer sequence  $\{a_n\}$ is called a divisibility sequence if $n|m$ implies $a_n |a_m$ (see Chapter 4 and Chapter 10 of \cite{Everest3} for a discussion of such sequences and see the discussion of the formal group of an elliptic curve in Chapter 4 of \cite{Sil1} for an explanation of why the denominators form a divisibility sequence).   We now take a closer look at these denominators.
\subsection{The Denominators of Points on an Elliptic Curve}
Let $E$ be an elliptic curve as in \eqref{eq:elliptic}, fix a point of infinite order $P_1=(x_1,y_1)$ on the curve and let $P_n=(x_n,y_n)=[n]P_1$ be the $n$-th multiple of $P_1$ for a non-zero integer $n$.  In the notation above, using properties of elliptic curves  one can show with various degrees of difficulty that the following statements are true:
\begin{itemize}
\item The same primes divide the (reduced) denominators of $x_n$ and $y_n$, and therefore we can speak about primes dividing the ``denominator'' of $P_n$.  (This follows from looking at the elliptic curve equation we use.)
\item For all $n$ sufficiently large in absolute value, it is the case that $P_n$ has a \emph{primitive divisor}, i.e. a prime dividing the ``denominator'' of $P_n$ but not the ``denominator'' of any $P_m$  with $|m| < |n|$.  Denote the largest primitive divisor of $P_n$ by $p_n$.  We call $p_n$ an \emph{an indicator prime}.  (This can be deduced from the rate of growth of denominators relative to the rate of growth of the exponents of the primes dividing the denominators.)
   \item  $(\denom(x_n),\denom(x_m))=\denom(x_{(m,n)})$ for sufficiently large $|m|$ and $|n|$ and therefore $\denom(x_m)$ divides $\denom(x_n)$ if and only if  $m$ divides $n$ (again for sufficiently large $|m|$ and $|n|$).  (This property uses the existence of primitive divisors as well as the fact that the denominators form a divisibility sequence.)
\item Poonen showed that the set  $\{p_{\ell}: \ell \mbox{ is a prime numer }\}$ has natural density equal to 0.  (This density result was proved using some results of Serre and was probably the most technically challenging part of Poonen's paper.)
\end{itemize}
          Now it is not hard to see that if $p_{\ell}$ is not allowed in the denominators of the elements of our big ring, then the coordinates of the point $[\ell]P_1=(x_{\ell},y_{\ell})$ will not be  solutions to the elliptic curve curve equation in our ring.  Further, we will also exclude \emph{all the multiples} of this point but will not affect the points whose indices are prime to $\ell$.   This is the principal mechanism for controlling which solutions to the elliptic curve equation appear in our ring.

In a 2009 paper (see \cite{EE}) Eisentr\"ager and Everest extended Poonen's method to prove a theorem concerning sets of complementary primes.  If $\calP$ is the set of all primes of $\Q$, then two subsets $\calT, \calS \subset \calP$ are exactly complementary if $\calS\cup \calT = \calP$ and $\calS \cap \calT=\emptyset$.   We state this theorem below.
\begin{theorem}
There are exactly complementary recursive sets $\calS$, $\calT \subset \calP$ such that
Hilbert's Tenth Problem is undecidable for both rings $O_{\Q,\calS}$ and $O_{\Q,\calT}$ .
\end{theorem}

\section{Old Dreams Die Hard}
So what about the old plan for modeling $\Z$ using  indices of elliptic curve points?  As it turns out the old plan can be resurrected but not exactly as intended.  We remind the reader that the stumbling block for that plan was showing that the set $\{(x_n,y_n), (x_m,y_m), (x_{nm},y_{nm})\}$ was Diophantine over $\Q$, where  $(x_n,y_n)$ are the affine coordinates of an $n$-th multiple of a generator of an elliptic curve of rank one over $\Q$.  (See Section \ref{subsec:oldplan}.)    We manage to make this set Diophantine but only over some very big rings.  To do this we modify Poonen's idea by not inverting \emph{any}  indicator primes $p_{\ell}$.  In this case no point of the elliptic curve has coordinates in our ring and we have to represent a point by a quadruple of numerators and denominators.  While this is more awkward, we can derive some benefits.
\subsection{Defining Multiplication of Indices}
The main result pertaining to index multiplication is stated below.
\begin{theorem}%
\label{thm:main}%
Let $E$ be an elliptic curve defined and of rank one over $\Q$.  Let $P$ be a
generator of $E(\Q)$ modulo the torsion subgroup, and fix an affine (Weierstrass) equation for $E$ of the form $y^2=x^3+ ax +b$, with $a, b \in \Z$.  If $(x_n,y_n)$ are the coordinates of $[n]P$ with  $n \not = 0$ derived from this (Weierstrass) equation, then there exists a set of primes $\calW$ of
natural density one,  and a positive integer $m_0$ such that the following set $\Pi \subset O_{\Q,\calW}^{12}$ is Diophantine over $O_{\Q,\calW}$.
\[%
\begin{array}{c}
(U_1, U_2, U_3, X_1,X_2,X_3, V_1,V_2, V_3, Y_1, Y_2, Y_3)\in \Pi \Leftrightarrow\\
\exists \mbox{ unique } k_1, k_2, k_3 \in \Z_{\not= 0}   \mbox{ such that } \\
\left(\frac{U_i}{V_i}, \frac{X_i}{Y_i}\right )=(x_{m_0k_i}, y_{m_0k_i}), \mbox{ for } i=1,2,3, \mbox{ and } k_3=k_1k_2.
\end{array}
\]%
\end{theorem}%
(See \cite{Sh38}.)

We outline the proof under some simplifying assumptions. Let $E$ as before be an elliptic curve of rank one defined over $\Q$ and let $E(\Q)$ be the set of all the points of the curve with rational coordinates together with $O$ -- the point at infinity.  For a generator $P$ of $E(\Q)$,  for $n \in \Z_{\not = 0}$, we let $[n]P = (x_n,y_n)=(\frac{U_n}{V_n},\frac{X_n}{Y_n}) \in \Q^2$, where $V_n>0, Y_n >0, (U_n,V_n) =1, (X_n,Y_n)=1$, $U_n, V_n, X_n, Y_n \in \Z$.
(We are still using a (Weierstrass) equation of our elliptic curve of the form $y^2=x^3+ax+b$.) We assume that every non-trivial multiple of the generator $P$ has an odd  primitive divisor, i.e. for every $n \in \Z_{\not=0,\pm1}$ there exists a prime $p \not=2$ such that $p | V_n$ but $p \not | V_m$ for any $m$ with $|m|<|n|$.  We also assume that coordinates of $P$ are integers. A point $P_n$ will now be represented by a quadruple $(A_n, B_n, C_n, D_n)$ of elements in our ring and this representation will not be unique because while it is possible to require that a numerator and the denominator are relatively prime in our big ring, big rings in general have infinitely many units and this will stand in a way of a unique representation of a point. In our big ring we do not invert the set
\[
\calV=\{\mbox{The largest odd primitive prime factor } p_{\ell^i} \mbox{ of } V_{\ell^i}\},
\]
where $\ell$ runs through all the prime numbers and $i \in \Z_{>0}$.  Not inverting these primes will ensure that $\denom(x_m)$ divides $\denom(x_n)$ if and only if $m$ divides $n$.  We invert \emph{all the other primes}. In the resulting big ring we now have that
\[
(V_m,V_n)=1, (m, V_n), (n,V_m) = 1, V_mV_n | V_k \land V_k |V_mV_n \Longrightarrow |k|=|mn|.
\]
More work is required to get rid of the absolute values and to remove the relative primeness condition, but that task can be accomplished with standard methods.
Given that each point of the elliptic curve has infinitely many quadruples representing it we cannot construct a Diophantine model of $\Z$  (or show that $\Z$ is existentially definably interpretable in $O_{\Q,\calW}$) as above.  What we can do is construct a class Diophantine model of $\Z$ defined below.  (The corresponding model-theoretic notion is existential interpretability.  Again see Chapter 1, Section 3 of \cite{Marker}.)
\begin{definition}[Class Diophantine Model]%
Let $R$ be a countable recursive ring, let $D \subset R^k$, $k \in \Z_{>0}$ be a Diophantine subset, and let $\approx$
be a (Diophantine) equivalence relation on $D$, i.e assume that the set $\{(\bar x,\bar y): \bar x,\bar y \in D, \bar x
\approx \bar y\}$ is  a Diophantine subset of $R^{2k}$. Let $D=\bigcup_{i \in \Z}D_i$, where $D_i$ is an
equivalence class of $\approx$, and let $\phi: \Z \longrightarrow \{D_i, i \in \Z\}$ be defined by
$\phi(i) = D_i$. Finally assume that the sets
\[%
Plus=\{(\bar x,\bar y,\bar z): \bar x \in D_i, \bar y\in D_j, \bar z \in D_{i+j}\}
\]%
and
\[%
Times =\{(\bar x,\bar y,\bar z): \bar x \in D_i, \bar y\in D_j, \bar z \in D_{ij}\}
\]%
 are Diophantine over $R$.
In this case we say that $R$ has a \emph{class} Diophantine model of $\Z$.
\end{definition}%
It is not hard to show that the rings with class Diophantine models of $\Z$ have undecidable Diophantine sets, just as the rings with Diophantine models of $\Z$. As as result of our ability to give a Diophantine definition of multiplication of indices we can construct a class Diophantine model of $\Z$ over a class of big subrings of $\Q$ not covered by the results of Poonen or Eisentr\"{a}ger and Everest.

\begin{corollary}%
\label{cor:main}
In the notation above, for $n \not = 0$ let $\phi(n)=[ (U_{m_0n},X_{m_0n}, V_{m_0n}, Y_{m_0n})]$, the equivalence class of $(U_{m_0n},X_{m_0n}, V_{m_0n}, Y_{m_0n})$ under the equivalence relation described below,   where $U_{m_0n},X_{m_0n}$, $V_{m_0n}, Y_{m_0n} \in O_{\Q,\calS}, V_{m_0n}Y_{m_0n} \not = 0$, and
$\displaystyle (x_{m_0n},y_{m_0n})=\left(\frac{U_{m_0n}}{V_{m_0n}},\frac{X_{m_0n}}{Y_{m_0n}} \right)$.  Let $\phi(0)=\{(0,0,0,0)\}$.  Then $\phi$  is a class Diophantine model of $\Z$. (Here  if  $V\hat V\hat Y Y\not=0$ we set $(U,X,V,Y) \approx (\hat{ U},\hat X, \hat V, \hat Y)$ if and only if $\displaystyle \frac{\hat{U}}{\hat{V}}=\frac{U}{V}$ and $\displaystyle \frac{{\hat X}}{{\hat Y}}=\frac{X}{Y}$.)
\end{corollary}%
We have now pretty much  covered the state of existential affairs and it is time to  return to the issue of lifting the ban on a few universal quantifiers.  One could interpret this as a sign of surrender in the face of the overwhelming enemy (i.e. HTP for $\Q$), but we prefer a more optimistic interpretation:  a gradual gathering of forces.

\section{Matters of the First Order or Back to the Future}
The result defining integers over $\Q$ using the "full force'' of the first-order language is pretty old and belongs to Julia Robinson.
\begin{theorem}[Julia Robinson]
$\Z$ is first-order definable over $\Q$. (See \cite{Rob1}.)
\end{theorem}
Julia Robinson used quadratic forms and Hasse-Minkwoski Theorem to prove her theorem and in the process produced an existential definition over $\Q$ of the set
\[
\mbox{Int}_q=\{x \in \Q: x =\frac{a}{b}, a, b \in \Z, b \not \equiv 0 \mod q\}
\]
for any given prime $q$.  It is this existential definition and the fact that we can ``simulate'' the denominators because we can define the set of non-zero elements of any small ring, that allowed us to conclude that $\Z$ was existentially definable over small rings.

In a 2007 paper Cornelissen and Zahidi analyzed Julia Robinson's formula and showed that it can be converted to a formula of the form $(\forall\exists\forall\exists)(F = 0)$
where the $\forall$-quantifiers run over a total of 8 variables, and where $F$ is a polynomial.  They also were the first ones to consider optimizing Julia Robinson's result using elliptic divisibility sequences:
\begin{theorem}[Cornelissen and Zahidi]
Assuming a (heuristically plausible) conjecture concerning denominators of points on elliptic curves  over $\Q$, there exists a first-order model of $\Z$ over $\Q$ using just one universal quantifier. (See \cite{CZ2}.)
\end{theorem}
And so a conjecture makes its appearance to help push a result along, though not yet a famous conjecture.

In a 2007 paper (see \cite{CS}), using elliptic curves along the lines of Poonen's method and results of Julia Robinson,  Cornelissen and the author showed that one could define  $\Z$ over a large subring of $\Q$ using two universal quantifiers.  Continuing further down this road, in 2008 Poonen in \cite{PO4} produced an unconditional improvement of the first-order definition of integers over $\Q$ and over some big rings.

\begin{theorem}
\begin{itemize}
\item  $\Z$ is definable over $\Q$ using just two universal quantifiers in a $\forall\exists$-formula.
\item For any $\varepsilon >0$, there exists a set of rational primes $\calW_{\Q}$ of natural density greater than $1-\varepsilon$ such that $\Z$ is definable using just one quantifier  in a $\forall\exists$-formula over $O_{\Q,\calW_{\Q}}$.
\end{itemize}
\end{theorem}
Poonen used quadratic forms, quaternions and the Hasse Norm Principle.  His definition of $\Z$ over $\Q$ is simple enough to be reproduced here:
the set $\Z$ equals the set of $t \in \Q$ for which the following formula is true over $\Q$:
\begin{align*}
(\forall a, b)(\exists a_1, a_2, a_3, a_4, b_1, b_2, b_3, b_4, x_1, x_2, x_3, x_4, y_1, y_2, y_3, y_4, n)\\
(a + a_1^2 + a_2^2 + a_3^2+ a_4^2)(b + b^2_1 + b^2_2 + b^2_3 + b^2_4 )\cdot\\
[(x^2_1- ax_2^2 - bx^2_3 + abx^2_4-1)^2 + (y^2_1 - ay_2^2 - by^2_3 + aby^2_4 -1)^2+\\
+ n^2+(n-1)^2 \ldots (n - 2309)^2 + (2x_1 + 2y_1 + n- t)^2]= 0
\end{align*}
Using existential definition of multiplication on indices one can also show the following.

\begin{theorem}%
There exists a set $\calW$ of primes of $\Q$ of natural density one such that $\Z$ is first-order definable over $O_{\Q,\calW}$ using just one universal
quantifier in a $\forall\exists$-formula. (See \cite{Sh38}.)
\end{theorem}%

Further  the following result was announced by J. Koenigsmann in February 2009.

\begin{theorem}
$\Z$ is first-order definable over $\Q$ using just one $\forall$-quantifier in a $\forall\exists$ - formula.
\end{theorem}
To summarize the discussion so far we can conceive of the following big definability project partly discussed already in Section \ref{subsec:between}.
\begin{question} For which big subrings of $\Q$ is the following  true?
\begin{itemize}
\item  $\Z$ is existentially definable.
\item $\Z$ has an existential model.
\item $\Z$ is definable using one universal quantifier.
\end{itemize}
\end{question}

\begin{remark}
If we start counting the number of quantifiers we use, definability results over the field will no longer automatically imply the analogous definability results for the subrings.  Thus Koenigsman's result for $\Q$ does not automatically answer the question about the big subrings in general.
\end{remark}

We now briefly consider the progress made on the other arguably most interesting problem in the area: the Diophantine (un)decidability of the rings of integers of number fields.

\section{Meanwhile in a Galaxy not Far Away}
We start with a review of some terms.
\begin{itemize}
\item A number field is a finite extension of $\Q$.
\item A totally real number field is a number field all of whose embeddings into its algebraic closure are real.
\item A ring of integers $O_K$ of a number field $K$ is the set of all elements of the number field satisfying monic irreducible polynomials over $\Z$ or alternatively the integral closure of $\Z$ in the number field.
\item A prime of a number field $K$ is a prime ideal of $O_K$.  If $x \not = 0$ and $x \in O_K$, then for any prime $\pp$ of $K$ there exists a non-negative integer $m$ such that $x \in \pp^m$ but $x \not \in \pp^{m+1}$.  We call $m$ the order of $x$ at $\pp$ and write $m = \ord_{\pp}x$.  If $y \in K$ and $y \not = 0$, we write $y =\frac{x_1}{x_2}$, where $x_1, x_2 \in O_K$ with $x_1x_2 \not = 0$, and define $\ord_{\pp}y = \ord_{\pp}x_1 - \ord_{\pp}x_2$.  This definition is not dependent on the choice of $x_1$ and $x_2$ which are of course not unique.  We define $\ord_{\pp}0 = \infty$ for any prime $\pp$ of $K$.
\item  Any prime ideal $\pp$ of $O_K$ is maximal and the residue classes of $O_K$ modulo $\pp$ form a field.  This field is always finite and its size (a power of a rational prime number) is call the norm of $\pp$ denoted by ${\mathbf N}\pp$.
\item If $\calW$ is a set of primes of $K$, its natural density is defined to be the following limit if it exists:
\[
\lim_{X \rightarrow \infty}\frac{\#\{\pp \in \calW, {\mathbf N}\pp\leq X\}}{\# \{{\mathbf N} \pp \leq X\}}
\]
\item  Let $K$ be a number field and let $\calW$ be a set of primes of $K$.  Let $O_{K,\calW}$ be the following subring of $K$.
\[%
 \{x \in K: \ord_{\pp} x \geq 0 \, \,\forall \pp \not \in \calW\}
\]%
If $\calW = \emptyset$, then $O_{K,\calW}=O_K$ -- the ring of integers of $K$.  If $\calW$ contains all the primes of $K$,  then $O_{K,\calW}=K$.  If $\calW$ is finite, we call the ring \emph{small} (or the ring of $\calW$-integers).  If $\calW$ is infinite, we call the ring \emph{large}, and if the natural density of $\calW$ is one, we call the ring ``very large''.  These rings are the counterparts of the ``in between'' subrings of $\Q$.
\end{itemize}
The state of knowledge concerning the rings of integers is summarized in the theorem below.
\begin{theorem}%
\label{thm:ext1}
$\Z$ is Diophantine and HTP is unsolvable  over the rings of integers of the following fields:%
\begin{itemize}
\item ``Most'' extensions of degree 4 of $\Q$, totally real number fields and their extensions of degree 2. (See \cite{Den1}, \cite{Den2}.)  Note that
these fields include all abelian extensions.
\item Number fields with exactly one pair of non-real embeddings (See \cite{Ph1} and \cite{Sh2}.)%
\item Any number field $K$ such that there exists an elliptic curve $E$ of positive rank defined over $\Q$ with $[E(K):E(\Q)] < \infty$. (See \cite{Po2}, \cite{Sh33}, \cite{Po3}.)%
\item Any number field $K$ such that there exists an elliptic curve of rank 1 over $K$ and an abelian variety over $\Q$ keeping its rank over $K$. (See \cite{CPZ}.)
\end{itemize}%
\end{theorem}%

We also have several results concerning big rings in number fields.  Note that in the big rings below we actually give a Diophantine definition of $\Z$.
\begin{theorem}
\label{thm:ext2}
Let $K$ be a number field satisfying one of the following conditions:
\begin{itemize}
\item $K$ is a totally real field.
\item $K$ is an extension of degree 2 of a totally real  field.
\item There exists an elliptic curve $E$ defined over $\Q$ such that $[E(K):E(\Q)]< \infty$.
\end{itemize}
Let $\varepsilon >0$ be given.  Then there exists a set $\calS$ of non-archimedean primes of $K$ such that
\begin{itemize}%
\item The natural density of $\calS$ is greater $\displaystyle 1-\frac{1}{[K:\Q]}-\varepsilon$.
\item $\Z$ is Diophantine over $O_{K,\calS}$.
\item HTP is unsolvable over $O_{K,\calS}$.
\end{itemize}
(See \cite{Sh1}, \cite{Sh6}, \cite{Sh3},  \cite{Sh33}, \cite{Sh37}.)
\end{theorem}%
Over very large subrings, as was the case over $\Q$, we have Diophantine models of $\Z$ only.  The first theorem is a number field version of Poonen's method for $\Q$.  However the situation is more complicated over a number field and instead of constructing a model of $\Z$ by ``approximation'', what is constructed here is a model of a subset of the rational integers over which one can construct a model of $\Z$.  In short, one constructs a ``model of a model''.

\begin{theorem}
\label{thm:ext3}
Let $K$ be a number field with a rank one elliptic curve.  Then there exist recursive sets of
$K$-primes ${\mathcal T}_1$ and ${\mathcal T}_2$, both of natural density zero and with an
empty intersection, such that for any set ${\mathcal S}$ of  primes of $K$  containing ${\mathcal
T}_1$ and avoiding  ${\mathcal T}_2$, $\Z$ has an existential model and Hilbert's Tenth Problem is unsolvable over $O_{K,\calS}$.
(See \cite{PS}.)
\end{theorem}%

There is also a version of index multiplication over number fields.  Here the strategy remains the same as over $\Q$.
\begin{theorem}%
\label{thm:ext4}
 Let $K$ be a number field.  Let $E$ be an elliptic curve defined and of rank one over $K$.  Let $P$ be a
generator of $E(K)$ modulo the torsion subgroup, and fix an affine Weierstrass equation for $E$ of the form $y^2=x^3+ ax +b$, with $a, b \in O_K$.  Let $(x_n,y_n)$ be the coordinates of $[n]P$ with $n \not = 0$ derived from this Weierstrass equation.  Then there exists a recursive set of $K$-primes $\calW_K$ of natural density one,  and a positive integer $m_0$ such that the following set $\Pi \subset O_{K,\calW_K}^{12}$ is Diophantine over $O_{K,\calW_K}$:
\[%
\begin{array}{c}
(U_1, U_2, U_3, X_1,X_2,X_3, V_1,V_2, V_3, Y_1, Y_2, Y_3)\in \Pi \Leftrightarrow\\
\exists \mbox{ unique } k_1, k_2, k_3 \in \Z_{\not= 0}   \mbox{ such that } \\
\left(\frac{U_i}{V_i}, \frac{X_i}{Y_i}\right )=(x_{m_0k_i}, y_{m_0k_i}), \mbox{ for } i=1,2,3, \mbox{ and } k_3=k_1k_2,
\end{array}
\]%
$\Z$ has a class Diophantine model, and HTP is unsolvable over this ring.\\
See \cite{Sh38}
\end{theorem}%

Theorems \ref{thm:ext1} -- \ref{thm:ext4} have an elliptic curve assumption in their statements.  These assumptions come in two flavors: an existence of a rank one elliptic curve over a field in question and an assumption concerning an elliptic curve of positive rank, not changing its rank in a finite extension.  The rank one assumption is pretty straightforward but the ``positive stable rank'' assumption can be modified.  It is not hard to show that it is enough to have a ``positive stable rank'' phenomenon for every cyclic extension of prime degree to obtain a Diophantine definition of $\Z$ over the ring of integers of any number field.  The reduction takes place in several steps.
\be
\item Let $K$ be a number field and let $O_K$ be the ring of integers of $K$.  Further, let $M$ be the Galois closure of $K$ over $\Q$ and let $O_M$ be the ring of integers of $M$.  Under these assumptions if $\Z$ has a Diophantine definition over $O_M$, then $\Z$ has a Diophantine definition over $O_K$.  (Thus we can consider Galois extensions of $\Q$ only.)
\item Let $M/\Q$ be a Galois extension of number fields with $O_M$ the rings of integers of $M$ respectively.  Let $E_1,\ldots, E_n$ be all the cyclic subextensions of $M$ with $O_{E_1},\ldots,O_{E_n}$ the rings of integers of $E_1,\ldots,E_n$ respectively.  Observe that $\bigcap_{i=1}^nE_i=\Q$ and $\bigcap_{i=1}^nO_{E_i}=\Z$ and therefore if each $O_{E_i}$ has a Diophantine definition over $O_M$, then $\Z$ has a Diophantine definition over $O_M$.  (Thus, it is enough to show that in every cyclic extension the ring of integers below has a Diophantine definition over the ring of integers above.)
\item If $E \subseteq H \subseteq M$ is a finite extension of number fields, $O_H$ has a Diophantine definition over $O_M$, and $O_E$ has a Diophantine definition over $O_H$, then $O_E$ has a Diophantine definition over $O_M$.  (Thus, it is enough to consider cyclic extensions of prime degree only.)
\ee
In the case of big rings the same kind of reductions also work, but an extra effort is required to make sure the sets of primes allowed in the denominators have the right density. For a general discussion of reductions of this sort see \cite{Sh16} and Chapter 2 of \cite{Sh34}.  \\

Unfortunately as of now we do not have unconditional results asserting the existence of required elliptic curves, but with a ``little help'' from this time a famous conjecture we do have the following recent result by Mazur and Rubin.

\begin{theorem}
Suppose L/K is a cyclic extension of prime degree of number
fields.  If the Shafarevich-Tate Conjecture is true for $K$, then there is an elliptic curve E over K with
rank(E(L)) = rank(E(K)) = 1. (See \cite{MR}.)
\end{theorem}
As discussed above, this theorem has quite a few consequences.
\begin{corollary}
 If the Shafarevich-Tate Conjecture is true for all number fields, then the following statements are true.
\begin{itemize}
\item $\Z$ has a Diophantine definition over the ring of integers of any number field $K$.
\item    For any number field $K$ and any $\varepsilon >0$,  there exists a set $\calS$ of non-archimedean primes of $K$ such that
the natural density of $\calS$ is greater $\displaystyle 1-\frac{1}{[K:\Q]}-\varepsilon$ and $\Z$ is Diophantine over $O_{K,\calS}$.
\item  For any number field $K$, there exists a set of primes $\calS$ of natural density 1 such that $\Z$ has a Diophantine model over $O_{K,\calS}$.
\end{itemize}
\end{corollary}

\section{Final Remarks}
This article touched only on a small part of the subject which grew out of Hilbert's Tenth Problem.  In particular, we did not discuss a great number of results on the analogs of HTP over different kinds of functions fields and infinite algebraic extensions, and also rings where the problem becomes decidable.  We refer the interested reader to the following surveys and collections for more information: \cite{Dencollection}, \cite{Po7} and \cite{Sh34}.


\begin{thebibliography}{10}

\bibitem{CPZ}
Gunther Cornelissen, Thanases Pheidas, and Karim Zahidi.
\newblock Division-ample sets and diophantine problem for rings of integers.
\newblock {\em Journal de Th\'eorie des Nombres Bordeaux}, 17:727--735, 2005.

\bibitem{CS}
Gunther Cornelissen and Alexandra Shlapentokh.
\newblock Defining the integers in large rings of number fields using one
  universal quantifier.
\newblock {\em Zapiski Nauchnykh Seminarov POMI}, 358:199 -- 223, 2008.

\bibitem{CZ}
Gunther Cornelissen and Karim Zahidi.
\newblock Topology of diophantine sets: Remarks on {M}azur's conjectures.
\newblock In Jan Denef, Leonard Lipshitz, Thanases Pheidas, and Jan Van~Geel,
  editors, {\em Hilbert's Tenth Problem: Relations with Arithmetic and
  Algebraic Geometry}, volume 270 of {\em Contemporary Mathematics}, pages
  253--260. American Mathematical Society, 2000.

\bibitem{CZ2}
Gunther Cornelissen and Karim Zahidi.
\newblock Elliptic divisibility sequences and undecidable problems about
  rational points.
\newblock {\em J. Reine Angew. Math.}, 613:1--33, 2007.

\bibitem{Da1}
Martin Davis.
\newblock Hilbert's tenth problem is unsolvable.
\newblock {\em American Mathematical Monthly}, 80:233--269, 1973.

\bibitem{Da2}
Martin Davis, Yuri Matiyasevich, and Julia Robinson.
\newblock Hilbert's tenth problem. {D}iophantine equations: Positive aspects of
  a negative solution.
\newblock In {\em Proc. Sympos. Pure Math.}, volume~28, pages 323-- 378. Amer.
  Math. Soc., 1976.

\bibitem{Den1}
Jan Denef.
\newblock Hilbert's tenth problem for quadratic rings.
\newblock {\em Proc. Amer. Math. Soc.}, 48:214--220, 1975.

\bibitem{Den2}
Jan Denef and Leonard Lipshitz.
\newblock Diophantine sets over some rings of algebraic integers.
\newblock {\em Journal of London Mathematical Society}, 18(2):385--391, 1978.

\bibitem{Dencollection}
Jan Denef, Leonard Lipshitz, Thanases Pheidas, and Jan Van~Geel, editors.
\newblock {\em Hilbert's tenth problem: relations with arithmetic and algebraic
  geometry}, volume 270 of {\em Contemporary Mathematics}.
\newblock American Mathematical Society, Providence, RI, 2000.
\newblock Papers from the workshop held at Ghent University, Ghent, November
  2--5, 1999.

\bibitem{EE}
Kirsten Eisentr{\"a}ger and Graham Everest.
\newblock Descent on elliptic curves and {H}ilbert's tenth problem.
\newblock {\em Proc. Amer. Math. Soc.}, 137(6):1951--1959, 2009.

\bibitem{Everest3}
Graham Everest, Alf van~der Poorten, Igor Shparlinski, and Thomas Ward.
\newblock {\em Recurrence sequences}, volume 104 of {\em Mathematical Surveys
  and Monographs}.
\newblock American Mathematical Society, Providence, RI, 2003.

\bibitem{Marker}
David Marker.
\newblock {\em Model theory}, volume 217 of {\em Graduate Texts in
  Mathematics}.
\newblock Springer-Verlag, New York, 2002.
\newblock An introduction.

\bibitem{Mate}
Yuri~V. Matiyasevich.
\newblock {\em Hilbert's tenth problem}.
\newblock Foundations of Computing Series. MIT Press, Cambridge, MA, 1993.
\newblock Translated from the 1993 Russian original by the author, With a
  foreword by Martin Davis.

\bibitem{M1}
Barry Mazur.
\newblock The topology of rational points.
\newblock {\em Experimental Mathematics}, 1(1):35--45, 1992.

\bibitem{M2}
Barry Mazur.
\newblock Questions of decidability and undecidability in number theory.
\newblock {\em Journal of Symbolic Logic}, 59(2):353--371, June 1994.

\bibitem{M4}
Barry Mazur.
\newblock Open problems regarding rational points on curves and varieties.
\newblock In A.~J. Scholl and R.~L. Taylor, editors, {\em Galois
  Representations in Arithmetic Algebraic Geometry}. Cambridge University
  Press, 1998.

\bibitem{MR}
Barry Mazur and Karl Rubin.
\newblock Ranks of twists of elliptic curves and {H}ilbert's {T}enth {P}roblem.
\newblock {\em Inventiones Mathematicae}, 181:541--575, 2010.

\bibitem{Ph1}
Thanases Pheidas.
\newblock Hilbert's tenth problem for a class of rings of algebraic integers.
\newblock {\em Proceedings of American Mathematical Society}, 104(2):611--620,
  1988.

\bibitem{Po3}
Bjorn Poonen.
\newblock Elliptic curves whose rank does not grow and {H}ilbert's {T}enth
  {P}roblem over the rings of integers.
\newblock Private Communication.

\bibitem{Po}
Bjorn Poonen.
\newblock Using elliptic curves of rank one towards the undecidability of
  {H}ilbert's {T}enth {P}roblem over rings of algebraic integers.
\newblock In C.~Fieker and D.~Kohel, editors, {\em Algorithmic Number Theory},
  volume 2369 of {\em Lecture Notes in Computer Science}, pages 33--42.
  Springer Verlag, 2002.

\bibitem{Po2}
Bjorn Poonen.
\newblock {H}ilbert's {T}enth {P}roblem and {M}azur's conjecture for large
  subrings of ${\Q}$.
\newblock {\em Journal of AMS}, 16(4):981--990, 2003.

\bibitem{Po7}
Bjorn Poonen.
\newblock Undecidability in number theory.
\newblock {\em Notices Amer. Math. Soc.}, 55(3):344--350, 2008.

\bibitem{PO4}
Bjorn Poonen.
\newblock Characterizing integers among rational numbers with a
  universal-existential formula.
\newblock {\em Amer. J. Math.}, 131(3):675--682, 2009.

\bibitem{PS}
Bjorn Poonen and Alexandra Shlapentokh.
\newblock Diophantine definability of infinite discrete non-archimedean sets
  and diophantine models for large subrings of number fields.
\newblock {\em Journal f{\"u}r die Reine und Angewandte Mathematik},
  2005:27--48, 2005.

\bibitem{Rob1}
Julia Robinson.
\newblock Definability and decision problems in arithmetic.
\newblock {\em Journal of Symbolic Logic}, 14:98--114, 1949.

\bibitem{Rogers}
Hartley Rogers.
\newblock {\em Theory of Recursive Functions and Effective Computability}.
\newblock McGraw-Hill, New York, 1967.

\bibitem{Sh38}
Alexandra Shlapentokh.
\newblock Using indices of points on an elliptic curve to construct a
  diophantine model of $\mathbb {Z}$ and define $\mathbb {Z}$ using one
  universal quantifier in very large subrings of number fields, including
  $\mathbb {Q}$.
\newblock arXiv:0901.4168v1 [math.NT].

\bibitem{Sh2}
Alexandra Shlapentokh.
\newblock Extension of {Hilbert}'s tenth problem to some algebraic number
  fields.
\newblock {\em Communications on Pure and Applied Mathematics}, XLII:939--962,
  1989.

\bibitem{Sh1}
Alexandra Shlapentokh.
\newblock Diophantine definability over some rings of algebraic numbers with
  infinite number of primes allowed in the denominator.
\newblock {\em Inventiones Mathematicae}, 129:489--507, 1997.

\bibitem{Sh6}
Alexandra Shlapentokh.
\newblock Defining integrality at prime sets of high density in number fields.
\newblock {\em Duke Mathematical Journal}, 101(1):117--134, 2000.

\bibitem{Sh16}
Alexandra Shlapentokh.
\newblock Hilbert's tenth problem over number fields, a survey.
\newblock In Jan Denef, Leonard Lipshitz, Thanases Pheidas, and Jan Van~Geel,
  editors, {\em Hilbert's Tenth Problem: Relations with Arithmetic and
  Algebraic Geometry}, volume 270 of {\em Contemporary Mathematics}, pages
  107--137. American Mathematical Society, 2000.

\bibitem{Sh3}
Alexandra Shlapentokh.
\newblock On diophantine definability and decidability in large subrings of
  totally real number fields and their totally complex extensions of degree 2.
\newblock {\em Journal of Number Theory}, 95:227--252, 2002.

\bibitem{Sh34}
Alexandra Shlapentokh.
\newblock {\em Hilbert's Tenth Problem: Diophantine Classes and Extensions to
  Global Fields}.
\newblock Cambridge University Press, 2006.

\bibitem{Sh33}
Alexandra Shlapentokh.
\newblock Elliptic curves retaining their rank in finite extensions and
  {H}ilbert's tenth problem for rings of algebraic numbers.
\newblock {\em Trans. Amer. Math. Soc.}, 360(7):3541--3555, 2008.

\bibitem{Sh37}
Alexandra Shlapentokh.
\newblock Rings of algebraic numbers in infinite extensions of {$\Bbb Q$} and
  elliptic curves retaining their rank.
\newblock {\em Arch. Math. Logic}, 48(1):77--114, 2009.

\bibitem{Sil1}
Joseph Silverman.
\newblock {\em The Arithmetic of Elliptic Curves}.
\newblock Springer Verlag, New York, New York, 1986.

\end{thebibliography}
\end{document}